\documentclass[12pt,twoside,reqno]{amsart}
\usepackage{amsmath}
\usepackage{amsfonts}
\usepackage{amssymb}
\usepackage{color}
\usepackage{mathrsfs}
\usepackage{cite}
\newcommand{\beqn}{\begin{equation}}
\newcommand{\eeqn}{\end{equation}}
\newcommand{\bean}{\begin{eqnarray}}
\newcommand{\eean}{\end{eqnarray}}
\DeclareMathAlphabet{\mathpzc}{OT1}{pzc}{m}{it}
\newtheorem{theorem}{Theorem}[section]

\newtheorem{lemma}[theorem]{Lemma}
\newtheorem{proposition}[theorem]{Proposition}

\numberwithin{equation}{section}
\begin{document}
\title{Classification of extinction profiles for a one-dimensional diffusive Hamilton-Jacobi equation with critical absorption}
\thanks{}
\author{Razvan Gabriel Iagar}
\address{Instituto de Ciencias Matem\'aticas (ICMAT), Nicolas Cabrera 13-15, Campus de Cantoblanco,
E--28049, Madrid, Spain}
\email{razvan.iagar@icmat.es}
\address{Institute of Mathematics of the
Romanian Academy, P.O. Box 1-764, RO-014700, Bucharest, Romania.}

\author{Philippe Lauren\c{c}ot}
\address{Institut de Math\'ematiques de Toulouse, UMR~5219, Universit\'e de Toulouse, CNRS \\ F--31062 Toulouse Cedex 9, France}
\email{laurenco@math.univ-toulouse.fr}


\date{\today}

\begin{abstract}
A classification of the behavior of the solutions $f(\cdot,a)$ to the ordinary differential equation $(|f'|^{p-2} f')' + f - |f'|^{p-1} = 0$ in $(0,\infty)$ with initial condition $f(0,a)=a$ and $f'(0,a)=0$ is provided, according to the value of the parameter $a>0$, the exponent $p$ ranging in $(1,2)$. There is a threshold value $a_*$ which separates different behaviors of $f(\cdot,a)$: if $a>a_*$ then $f(\cdot,a)$ vanishes at least once in $(0,\infty)$ and takes negative values while $f(\cdot,a)$ is positive in $(0,\infty)$ and decays algebraically to zero as $r\to\infty$ if $a\in (0,a_*)$. At the threshold value, $f(\cdot,a_*)$ is also positive in $(0,\infty)$ but decays exponentially fast to zero as $r\to\infty$. The proof of these results relies on a transformation to a first-order ordinary differential equation and a monotonicity property with respect to $a>0$. This classification is one step in the description of the dynamics  near the extinction time of a diffusive Hamilton-Jacobi equation with critical gradient absorption and fast diffusion.
\end{abstract}

\maketitle

%
%
\pagestyle{myheadings}
\markboth{\sc{R.G. Iagar \& Ph. Lauren\c cot}}{\sc{Classification of extinction profiles}}

\section{Introduction}

Let $p\in (1,2)$. Owing to its scale invariance, the diffusive
Hamilton-Jacobi equation
\begin{equation}
\partial_t u - \partial_x \left( |\partial_x u|^{p-2} \partial_x u \right) +  |\partial_x u|^{p-1} = 0 \ , \qquad (t,x)\in (0,\infty)\times \mathbb{R}\ , \label{i1}
\end{equation}
is expected to have self-similar solutions with separate variables, that is, solutions of the form
\begin{equation}
u_s(t,x) = \left( (2-p) (T-t)_+ \right)^{1/(2-p)} f(|x|)\ , \qquad (t,x)\in (0,\infty)\times \mathbb{R}\ , \label{i2}
\end{equation}
which vanish identically after a finite time $T>0$. Inserting this ansatz in \eqref{i1} leads us to the ordinary differential equation
\begin{equation}
(|f'|^{p-2} f')' + f - |f'|^{p-1} = 0\ , \qquad r\in (0,\infty)\ , \label{i3}
\end{equation}
along with the boundary condition $f'(0)=0$ stemming from the
assumed symmetry and the expected smoothness of $u_s$ with respect
to the space variable. It is then natural to investigate the
behavior of solutions to \eqref{i3} according to the initial value
$f(0)$. The main motivation for such an analysis is that
non-negative self-similar solutions of the form \eqref{i2} are
expected to provide an accurate description of the behavior near the
extinction time of non-negative solutions to \eqref{i1} which enjoy
the finite time extinction property. Indeed, it follows from
\cite[Theorem~1.2]{IaLa12} that there are many non-negative
solutions to \eqref{i1} satisfying the latter property. The
classification of solutions to \eqref{i3} performed below allows us
to identify the behavior at the extinction time of non-negative
solutions to \eqref{i1} in the companion paper \cite{IaLaxx}, the
initial data being even in $\mathbb{R}$, non-increasing on
$(0,\infty)$ and decaying sufficiently rapidly as $x\to\infty$.

More precisely the main result of this paper is the following classification:

\begin{theorem}\label{thm1}
Given $a>0$ there is a unique solution $f(\cdot,a)$ to the initial value problem
\begin{eqnarray}
& & (|f'|^{p-2} f')' + f - |f'|^{p-1} = 0\ , \qquad r\in (0,\infty)\ , \label{i4} \\
& & f(0,a) = a\ , \quad f'(0,a)=0\ , \label{i5}
\end{eqnarray}
and
\begin{equation}
R(a) := \inf\left\{ r>0\ : \ f(r,a) = 0 \right\} \in (0,\infty]\ . \label{i5b}
\end{equation}
Furthermore there is $a_*>0$ with the following properties:
\begin{itemize}
\item[(a)] if $a>a_*$ then $R(a)<\infty$, $f(R(a),a)=0$, and $f'(R(a),a)<0$.
\item[(b)] if $a=a_*$ then $R(a_*)=\infty$ and there is $\ell_*>0$ such that
$$
\lim_{r\to\infty} e^{r/(p-1)} f(r,a_*) = \ell_*\ .
$$
\item[(c)] if $a\in (0,a_*)$ then $R(a)=\infty$ and
$$
\lim_{r\to\infty} r^{(2-p)/(p-1)} f(r,a) = \left( \frac{p-1}{2-p} \right)^{(2-p)/(p-1)}\ .
$$
\end{itemize}
\end{theorem}

Before giving a rough account of the proof of Theorem~\ref{thm1},
let us complete the discussion started before the statement of
Theorem~\ref{thm1} on the role of self-similar solutions to
\eqref{i1} of the form \eqref{i2} in the description of the dynamics
of non-negative solutions to \eqref{i1} near their extinction time.
According to Theorem~\ref{thm1} we have infinitely many non-negative
self-similar solutions of the form \eqref{i2} (corresponding to
$a\in (0,a_*]$), but it turns out that \emph{only one is selected}
by the dynamics of \eqref{i1} as the behavior near the extinction
time. More precisely, as shown in \cite{IaLaxx}, if $u$ is a
solution to \eqref{i1} emanating from a non-negative even initial
condition which is non-increasing on $(0,\infty)$ and decays
sufficiently rapidly as $x\to\infty$ and if $T_e$ denotes its
extinction time, then $u(t,x)$ behaves as
$((2-p)(T_e-t)_+)^{1/(2-p)} f(|x|,a_*)$ as $t\to T_e$. Let us point
out that this \textsl{universal} behavior is also true in higher
space dimensions $N\ge 2$ for $p\in (2N/(N+1),2)$, but the
identification of the corresponding self-similar profile is more
involved and requires completely different arguments \cite{IaLaxx}.
We also point out that a similar dynamics as the one described above
is observed for the fast diffusion equation
$$
\partial_t v - \Delta v^m + v^m = 0\ , \qquad (t,x)\in (0,\infty)\times \mathbb{R}^N\ ,
$$
when $m\in ((N-2)_+/N,1)$ \cite{FeVa01, dPSa02}.

Let us now describe more precisely the proof of Theorem~\ref{thm1}.
Given $a>0$, classical results guarantee the well-posedness of
\eqref{i4}-\eqref{i5}, see Section~\ref{s2}. In addition, there is
$R(a)\in (0,\infty]$ such that $f(\cdot,a)$ is a decreasing
one-to-one function from $[0,R(a))$ onto $(0,a]$. This property
allows us to introduce $\psi(\cdot,a)$ defined on $(0,1)$ by
\begin{equation}
\psi\left( 1 - \frac{f(r,a)}{a} , a\right) := \frac{|f'(r,a)|^p}{a^p}\ , \qquad r\in [0,R(a))\ . \label{i6}
\end{equation}
Thanks to \eqref{i4}-\eqref{i5}, the function $\psi(\cdot,a)$ solves
\begin{equation}
\begin{split}
& \psi'(y) + \frac{p}{p-1} \psi(y)^{(p-1)/p}(y) = \frac{p}{p-1} a^{2-p} (1-y)\ , \qquad y\in (0,1)\ , \\
& \psi(0) = 0\ .
\end{split} \label{i7}
\end{equation}
The transformation \eqref{i6} thus reduces the second-order
differential equation \eqref{i4} to the first-order differential
equation \eqref{i7}, which is already a valuable feature, but it
also has the very interesting property that $\psi(\cdot,a)$ is
monotone with respect to $a$. The latter is in particular of utmost
importance to investigate uniqueness issues, see 
\cite{CQW03, Shi04, IaLa13a, IaLa13b, YeYi15} for instance, where monotonicity with respect to the shooting parameter is used to establish uniqueness of the ``fast orbit'' for related problems. In addition, the
finiteness of $R(a)$ as well as the behavior of $f(r,a)$ as
$r\to\infty$ when $R(a)=\infty$ are directly connected to the
behavior of $\psi(y,a)$ as $y\to 1$. The core of the analysis is
actually the identification of the behavior of $\psi(y,a)$ as $y\to
1$ according to the value of $a$ and is performed in
Section~\ref{s3}. Interpreting the results obtained in
Section~\ref{s3} in terms of $f(\cdot,a)$ is done in
Section~\ref{s4}, where we prove Theorem~\ref{thm1}.

We end this introduction with a couple of remarks: on the one hand,
the approach developed in this paper does not seem to extend to the
study and classification of self-similar solutions to \eqref{i1} of
the form \eqref{i2} in several space dimensions, the main reason
being that the variable $r=|x|$ remains in the equation satisfied by
$\psi$. Indeed, it seems that no transformation similar to
\eqref{i6} is available in dimension $N\geq2$. Still, it is possible
to establish a result similar to Theorem~\ref{thm1} in higher space
dimensions but completely different arguments are used
\cite{IaLaxx}. On the other hand, there is a \emph{striking
difference} between \eqref{i3} and
\begin{equation}
(|f'|^{p-2} f')' + f - |f|^{p-2}f = 0\ , \qquad r\in (0,\infty)\ ,
\label{i8}
\end{equation}
which involves only zero order reaction terms. Indeed, in general,
\eqref{i8} and its generalizations have only one non-negative
$C^1$-smooth solution which is defined on $(0,\infty)$ and converges
to zero as $r\to\infty$, the so-called \textsl{ground state
solution}, see \cite{Kw89, SeTa00, ShWa16, Ya91b} and the references
therein. This is in sharp contrast with \eqref{i3} for which
infinitely many ground states exist, see Theorem~\ref{thm1}, but a
single one features a faster decay as $r\to\infty$. This
multiplicity of course complicates the analysis, as it requires not
only to identify the possible decay rates as $r\to\infty$, but also
the corresponding ranges of the parameter $a$.

\section{Well-posedness of \eqref{i4}-\eqref{i5}}\label{s2}

We begin with the well-posedness of \eqref{i4}-\eqref{i5} and basic properties of its solutions.

\begin{lemma}\label{lemb1}
Given $a>0$, there is a unique solution $f(\cdot,a)\in
C^1([0,\infty))$ to \eqref{i4}-\eqref{i5} such that $|f'|^{p-2}f'\in
C^1([0,\infty))$. Furthermore,
$$
R(a) = \inf\left\{ r>0\ : \ f(r,a) = 0 \right\} \in (0,\infty]
$$
and $f(\cdot,a)$ enjoys the following properties:
\begin{equation}
0 < f(r,a)<a \;\text{ and }\; - \left( a (1-e^{-r}) \right)^{1/(p-1)} < f'(r,a) < 0\ , \qquad r\in (0,R(a))\ , \label{b1}
\end{equation}
and
\begin{equation}
\frac{d}{dr} \left( e^r |f'(r,a)|^{p-2} f'(r,a) \right) = - e^r f(r,a)\ , \qquad r\in (0,R(a))\ . \label{b1b}
\end{equation}
\end{lemma}

\begin{proof}
Since $p\in (1,2)$, the Cauchy-Lipschitz theorem ensures the
existence and uniqueness of a solution $(f,g)\in
C^1([0,\mathcal{R}(a));\mathbb{R}^2)$ to the initial value problem
\begin{equation}
\begin{split}
& f'(r) = - |g(r)|^{(2-p)/(p-1)} g(r)\ , \qquad g'(r) = - |g(r)| + f(r)\ , \qquad r\in (0,\mathcal{R}(a))\ , \\
& f(0)=a\ , \ g(0) = 0\ ,
\end{split}\label{b2}
\end{equation}
where $\mathcal{R}(a)\in (0,\infty]$ is such that either
$\mathcal{R}(a)=\infty$ or
\begin{equation}
\mathcal{R}(a)<\infty \;\text{ and }\; \limsup_{r\to\mathcal{R}(a)} \left( |f(r)| + |g(r)| \right) = \infty\ . \label{b3}
\end{equation}
Since $g(r) = -|f'(r)|^{p-2} f'(r)$ by \eqref{b2} for $r\in
[0,\mathcal{R}(a))$, it readily follows from \eqref{b2} that $f$
solves \eqref{i4}-\eqref{i5}. A further consequence of \eqref{b2} is
that
\begin{align*}
\frac{d}{dr} \left[ \frac{p-1}{p} |g|^{p/(p-1)} + \frac{1}{2} f^2 \right] & = |g|^{(2-p)/(p-1)} g (f-|g|) - |g|^{(2-p)/(p-1)} g f \\
& = - |g|^{1/(p-1)} g \\
& \le \frac{p}{p-1} \left[ \frac{p-1}{p} |g|^{p/(p-1)} + \frac{1}{2} f^2 \right]\ ,
\end{align*}
which excludes the occurrence of \eqref{b3}. Therefore $\mathcal{R}(a)=\infty$ and the positivity of $a$ along with the continuity of $f$ guarantee that $R(a)>0$.

We next infer from \eqref{i4}-\eqref{i5} that
$$
\lim_{r\to 0} (|f'|^{p-2} f')'(r) = -a < 0\ ,
$$
which implies that $f'$ is negative in a right neighborhood of $r=0$ as $f'(0)=0$. Using again \eqref{i4} we note that
\begin{equation}
\frac{d}{dr} \left( e^r |f'(r)|^{p-2} f'(r) \right) =  e^r \left[ |f'(r)|^{p-2} f'(r) + |f'(r)|^{p-1} - f(r) \right]\ , \qquad r>0\ . \label{b4}
\end{equation}
Consequently, as long as $f'(r)$ is negative and $r\in (0,R(a))$, there holds
$$
\frac{d}{dr} \left( e^r |f'(r)|^{p-2} f'(r) \right) =  - e^r f(r) < 0\ ,
$$
from which we deduce that $f'$ cannot vanish in $(0,R(a))$. We have thus proved that $f'(r)<0$ and $f(r)\in (0,a)$ for $r\in (0,R(a))$ as well as \eqref{b1b}. Combining these properties gives
$$
- \frac{d}{dr} \left( e^r |f'(r)|^{p-1} \right) \ge - a e^r\ , \qquad r\in (0,R(a))\ ,
$$
hence, after integration and using \eqref{i5},
$$
- e^r |f'(r)|^{p-1} \ge - a (e^r-1)\ , \qquad r\in (0,R(a))\ .
$$
This completes the proof of Lemma~\ref{lemb1}.
\end{proof}

\section{An alternative formulation}\label{s3}

Let $a>0$ and set $f=f(\cdot,a)$. As $f'<0$ in $(0,R(a))$ by
\eqref{b1}, the function $a-f$ is an increasing one-to-one function
from $[0,R(a))$ onto $[0,a)$ and we denote its inverse by $F$. Then
$F$ is an increasing function from $[0,a)$ onto $[0,R(a))$ and we
may define
\begin{equation}
\psi(y) = \psi(y,a) := \frac{1}{a^p} |f'(F(ay))|^p\ , \qquad y\in [0,1)\ . \label{c1}
\end{equation}
Equivalently,
\begin{equation}
\psi\left( 1 - \frac{f(r)}{a} \right) = \frac{|f'(r)|^p}{a^p}\ ,
\qquad r\in [0,R(a))\ , \label{c2}
\end{equation}
and
\begin{equation}
\psi'\left( 1 - \frac{f(r)}{a} \right) = - \frac{p}{(p-1) a^{p-1}}
\left( |f'|^{p-2} f' \right)'(r)\ , \qquad r\in [0,R(a))\ .
\label{c3}
\end{equation}
We then infer from \eqref{i4}-\eqref{i5}, \eqref{c2}, and \eqref{c3} that $\psi$ solves
\begin{eqnarray}
& & \psi'(y) + \frac{p}{p-1} \psi(y)^{(p-1)/p} = \frac{p a^{2-p}}{p-1} (1-y)\ , \qquad y\in (0,1)\ , \label{c4} \\
& & \psi(0) = 0\ . \label{c5}
\end{eqnarray}
We also deduce from \eqref{c4}-\eqref{c5} that
\begin{equation}
\psi'(0) = \frac{p a^{2-p}}{p-1}>0\ . \label{c6}
\end{equation}

\subsection{Comparison and monotonicity}\label{s3.1}

Though the equation \eqref{c4} involves the exponent $(p-1)/p$, which ranges in $(0,1)$, the following comparison principle is available:

\begin{lemma}[Comparison principle]\label{lemc0}
Let $\xi_i\in C^1([0,1))$, $i=1,2$, be two functions satisfying $\xi_1(0) \le \xi_2(0)$ and
\begin{equation}
\xi_1'(y) + \frac{p}{p-1} \xi_1(y)^{(p-1)/p} \le \xi_2'(y) + \frac{p}{p-1} \xi_2(y)^{(p-1)/p}\ , \qquad y\in (0,1)\ . \label{c0}
\end{equation}
Then $\xi_1(y) \le \xi_2(y)$ for $y\in [0,1)$.
\end{lemma}

\begin{proof}
Lemma~\ref{lemc0} actually follows from the monotonicity of $z\mapsto z^{(p-1)/p}$ and we recall its proof for the sake of completeness. Let $\delta>0$ and define
$$
y_\delta := \inf\{ y\in [0,1)\ :\ \xi_1(y)=\xi_2(y)+\delta\}\ .
$$
Clearly $y_\delta>0$ since $\xi_1(0)-\xi_2(0)-\delta \le - \delta <
0$. Assume for contradiction that $y_\delta<1$. Then
$\xi_1-\xi_2-\delta<0$ in $[0,y_\delta)$ and $(\xi_1' -
\xi_2')(y_\delta)\ge 0$, while \eqref{c0} gives
\begin{align*}
(\xi_1'-\xi_2')(y_\delta) & \le \frac{p}{p-1} \xi_2(y_\delta)^{(p-1)/p} - \frac{p}{p-1} \xi_1(y_\delta)^{(p-1)/p} \\
& = \frac{p}{p-1} \xi_2(y_\delta)^{(p-1)/p} - \frac{p}{p-1} \left( \xi_2(y_\delta) + \delta \right)^{(p-1)/p} < 0\ ,
\end{align*}
and a contradiction. Consequently, $\xi_1\le \xi_2+\delta$ in
$[0,1)$ and, since this inequality is valid for any $\delta>0$, we
conclude that $\xi_1\le \xi_2$ in $[0,1)$.
\end{proof}

The transformation \eqref{c1} has thus reduced the second-order
equation \eqref{i4} to the first-order equation \eqref{c4}, which
lowers the complexity of the problem. An additional property, which
turns out to be of high interest as well, of solutions to
\eqref{c4}-\eqref{c5} is their monotonicity with respect to $a$,
which is obviously a simple consequence of the comparison principle
established in Lemma~\ref{lemc0}. A more precise result is actually
available.

\begin{lemma}[Monotonicity with respect to $a$]\label{lemc1}
Consider $0<a_1<a_2$. Then there exists $K(p)>0$ depending only on
$p$ such that, for $y\in [0,1)$,
\begin{eqnarray*}
\psi(y,a_1) \le \psi(y,a_2) & \le & \psi(y,a_1) + K(p) (a_2-a_1)^{2-p}\ , \\
|\psi'(y,a_1) - \psi'(y,a_2)| & \le & K(p) \left[ (a_2-a_1)^{(2-p)(p-1)/p} + (a_2-a_1)^{2-p} \right]\ .
\end{eqnarray*}
In addition, $\psi(y,a_1)<\psi(y,a_2)$ for any $y\in(0,1)$.
\end{lemma}

\begin{proof}
Set $\psi_i=\psi(\cdot,a_i)$, $i=1,2$. Since $a_1<a_2$, it readily
follows from \eqref{i4}-\eqref{i5} that we can apply
Lemma~\ref{lemc0} with $(\xi_1,\xi_2)=(\psi_1,\psi_2)$.
Consequently, $\psi_1\le \psi_2$ in $[0,1)$.

We next put $M:=p \left( a_2^{2-p} - a_1^{2-p} \right)/(p-1)$ and
$\xi_2(y)=\psi_1(y)+My$ for $y\in [0,1)$. Then
$\xi_2(0)=0=\psi_2(0)$ and it follows from \eqref{c4} that, for
$y\in (0,1)$,
\begin{align*}
\xi_2'(y) + \frac{p}{p-1} \xi_2(y)^{(p-1)/p} & \ge \psi_1'(y) + M +\frac{p}{p-1} \psi_1(y)^{(p-1)/p} \\
& \ge M(1-y) + \frac{p a_1^{2-p}}{p-1} (1-y) = \frac{p a_2^{2-p}}{p-1} (1-y) \\
& = \psi_2'(y) + \frac{p}{p-1} \psi_2(y)^{(p-1)/p}\ .
\end{align*}
Applying Lemma~\ref{lemc0} to $(\xi_1,\xi_2) = (\psi_2,\xi_2)$
entails that $\psi_2\le \xi_2$ in $[0,1)$, which completes the proof
of the first statement of Lemma~\ref{lemc1}. We next infer from
\eqref{c4}, the H\"older continuity of $z\mapsto z^{(p-1)/p}$, and the first statement of Lemma~\ref{lemc1} that
\begin{align*}
|\psi_1'(y) -\psi_2'(y)| & \le \frac{p}{p-1} |\psi_1(y) - \psi_2(y)|^{(p-1)/p} + \frac{p}{p-1} (a_2-a_1)^{2-p} \\
& \le \frac{p}{p-1}  K(p)^{(p-1)/p} (a_2-a_1)^{(2-p)(p-1)/p} +
\frac{p}{p-1} (a_2-a_1)^{2-p}\ ,
\end{align*}
and thus complete the proof of the continuous dependence with
respect to $a$.

Finally, since $a_1<a_2$, it follows that
$$\bar{y}:=\sup\{y\in(0,1): \psi_1(z)<\psi_2(z) \ {\rm for} \ z\in(0,y)\}>0.$$
Assume for contradiction that $\bar{y}\in(0,1)$. Then
$\psi_2(\bar{y})=\psi_1(\bar{y})$ and, since $\psi_2\geq\psi_1$ in
$(0,1)$, then $\bar{y}$ is a point of minimum for $\psi_2-\psi_1$,
so that $(\psi_2-\psi_1)'(\bar{y})=0$. We infer from \eqref{c4} that
$$
0=(\psi_2-\psi_1)'(\bar{y})+\frac{p}{p-1}\left[\psi_2^{(p-1)/p}(\bar{y})-\psi_1^{(p-1)/p}(\bar{y})\right]=\frac{p}{p-1}(a_2^{2-p}-a_1^{2-p})(1-\bar{y}),
$$
which leads to $a_1=a_2$, hence a contradiction. This proves that
$\bar{y}=1$ and thereby completes the proof of Lemma~\ref{lemc1}.
\end{proof}

\subsection{Behavior of $\psi(y,a)$ as $y\to 1$}\label{s3.2}

We next describe the shape of $\psi(\cdot,a)$.

\begin{lemma}\label{lemc2}
Given $a>0$ there is $y_a\in (0,1)$ such that
\begin{equation}
\psi'(y_a,a)=0\ , \quad \psi'(y,a) (y-y_a)<0\ , \qquad y\in (0,1)\setminus\{y_a\}\ . \label{c7}
\end{equation}
Moreover there is $\ell(a)\ge 0$ such that
\begin{equation}
\lim_{y\to 1} \psi(y,a) = \ell(a)\ , \label{c8}
\end{equation}
and
\begin{equation}
\psi(y,a) \ge a^{p(2-p)/(p-1)} (1-y)^{p/(p-1)} \ , \qquad y\in (y_a,1) \ . \label{c8b}
\end{equation}
\end{lemma}

\begin{proof}
We define $y_a := \inf\{ y\in (0,1)\ :\ \psi'(y)=0\}$ and note that $y_a>0$ by \eqref{c6}. Assume for contradiction that $y_a=1$. Then $\psi'>0$ in $[0,1)$ and it follows from \eqref{c1} and \eqref{c4} that
$$
0 \le \psi(y)^{(p-1)/p} \le a^{2-p} (1-y)\ , \qquad y\in (0,1)\ .
$$
Consequently, $\psi(1)=0=\psi(0)$ which contradicts the strict monotonicity of $\psi$. Therefore $y_a\in (0,1)$ with $\psi'>0$ in $[0,y_a)$, $\psi'(y_a)=0$, and
$$
\psi''(y_a) = - \psi(y_a)^{-1/p} \psi'(y_a) - \frac{p a^{2-p}}{p-1} = - \frac{p a^{2-p}}{p-1} <0\ .
$$
In particular, $\psi'$ is negative in a right neighborhood of $y_a$.
Assume for contradiction that there is $z\in (y_a,1)$ such that
$\psi'(y)<0$ for $y\in (y_a,z)$ and $\psi'(z)=0$. Then $\psi''(z)\ge
0$, while \eqref{c4} entails that $\psi''(z) = -p a^{2-p}/(p-1)<0$,
and a contradiction. We have thus proved \eqref{c7} which, together
with \eqref{c1}, implies in particular that $\psi$ is positive and
decreasing on $(y_a,1)$, hence \eqref{c8}.

Finally, if $y\in [y_a,1)$, one has $\psi'(y)<0$ by \eqref{c7} and we infer from \eqref{c4} that
$$
\frac{p}{p-1} \psi(y)^{(p-1)/p} \ge \frac{p a^{2-p}}{p-1} (1-y)\ ,
$$
from which \eqref{c8b} readily follows.
\end{proof}

The next step, which is the cornerstone of the classification of the behavior of $\psi(\cdot,a)$ according to the value of $a$, is to elucidate the behavior of $\psi(y,a)$ as $y\to 1$. While it is obvious if $\ell(a)>0$, more information is needed when $\ell(a)=0$.

\begin{lemma}\label{lemc3}
Let $a>0$ and assume that $\ell(a)=0$. Then $y\mapsto \psi(y,a) (1-y)^{-p}$ has a limit as $y\to 1$ and
\begin{eqnarray}
& & 0 \le \psi(y,a) \le \kappa (1-y)^p\ , \qquad y\in (0,1)\ , \label{c9} \\
& & \lim_{y\to 1} \psi(y,a) (1-y)^{-p} \in \{ 0 , \kappa \}\ , \label{c10}
\end{eqnarray}
where $\kappa := (p-1)^{-p}$.
\end{lemma}

\begin{proof}
It readily follows from \eqref{c4} that, for $y\in (0,1)$,
\begin{align*}
\left( \psi^{1/p} \right)'(y) + \frac{1}{p-1} & = \frac{1}{p} \psi(y)^{-(p-1)/p} \psi'(y) + \frac{1}{p-1} \\
& = \frac{a^{2-p}}{p-1} (1-y) \psi(y)^{-(p-1)/p} \ge 0\ .
\end{align*}
Integrating the above differential inequality over $(y,1)$ and using $\ell(a)=0$ lead us to
$$
\frac{1}{p-1} \ge \psi(y)^{1/p} + \frac{y}{p-1}\ , \qquad y\in (0,1)\ ,
$$
hence \eqref{c9}.

We next define
\begin{equation}
\varphi(y) = \varphi(y,a) := \psi(y,a) (1-y)^{-p}\ , \qquad y\in [0,1)\ , \label{g}
\end{equation}
and deduce from \eqref{c4}-\eqref{c5} that $\varphi$ solves
\begin{eqnarray}
& & \varphi'(y) + \frac{p}{1-y} \varphi(y)^{(p-1)/p} \left( \kappa^{1/p} - \varphi(y)^{1/p} \right) = \frac{p a^{2-p}}{p-1} (1-y)^{1-p}\ , \qquad y\in (0,1)\ , \label{c11} \\
& & \varphi(0) = 0\ . \label{c12}
\end{eqnarray}
Integrating \eqref{c13} over $(0,y)$ and using \eqref{c12} give
\begin{equation}
\varphi(y) + p \int_0^y \Phi(z)\ dz = \frac{p a^{2-p}}{(p-1)(2-p)} \left[ 1 - (1-y)^{2-p} \right] \label{c13}
\end{equation}
for $y\in [0,1)$, where
$$
\Phi(y) := \varphi(y)^{(p-1)/p} \left[  \frac{\kappa^{1/p} - \varphi(y)^{1/p}}{1-y} \right]\ , \qquad y\in [0,1)\ .
$$
We then infer from \eqref{c9} that $\Phi\ge 0$ in $(0,1)$, which
gives, together with \eqref{c13} and the non-negativity of
$\varphi$,
$$
0 \le p \int_0^y \Phi(z)\ dz \le \frac{p a^{2-p}}{(p-1)(2-p)} \ , \qquad y\in [0,1)\ .
$$
Consequently, $\Phi\in L^1(0,1)$ and \eqref{c13} ensures that $\varphi(y)$ has a limit $L$ as $y\to 1$ given by
$$
\lim_{y\to 1} \varphi(y) = L := \frac{p a^{2-p}}{(p-1)(2-p)} - p \int_0^1 \Phi(y)\ dy\ .
$$
Recalling the definition of $\Phi$, we realize that
$$
\lim_{y\to 1} (1-y) \Phi(y) = L^{(p-1)/p} \left( \kappa^{1/p} - L^{1/p} \right)\ ,
$$
and the integrability of $\Phi$ implies that $L\in \{0,\kappa\}$.
\end{proof}

\subsection{Classification}\label{s3.3}

The outcome of Lemma~\ref{lemc2} and Lemma~\ref{lemc3} allows us to split the range of $a$ into three sets according to the behavior of $\psi(y,a)$ as $y\to 1$. More precisely, we define
\begin{eqnarray*}
\mathcal{A} & := & \left\{ a\in (0,\infty)\ : \ \ell(a)>0 \right\}\ , \\
\mathcal{B} & := & \left\{ a\in (0,\infty)\ : \ \lim_{y\to 1} \psi(y,a) (1-y)^{-p} = \kappa \right\}\ , \\
\mathcal{C} & := & \left\{ a\in (0,\infty)\ : \ \lim_{y\to 1} \psi(y,a) (1-y)^{-p} = 0 \right\}\ .
\end{eqnarray*}
Indeed, according to Lemma~\ref{lemc2} and Lemma~\ref{lemc3}, the sets $\mathcal{A}$, $\mathcal{B}$, and $\mathcal{C}$ are disjoint and
$$
\mathcal{A} \cup \mathcal{B} \cup \mathcal{C} = (0,\infty)\ .
$$

We now provide a more accurate description of these sets and begin with $\mathcal{A}$.

\begin{lemma}\label{lemc4}
There holds
\begin{equation}
a \in \mathcal{A} \ {\rm if \ and \ only \ if }\; \ \sup_{y\in
[0,1)} \left\{ \psi(y,a) (1-y)^{-p} \right\} > \kappa\ . \label{c15}
\end{equation}
Furthermore, there is $a^*>0$ such that $\mathcal{A}=(a^*,\infty)$.
\end{lemma}

\begin{proof}
As in the proof of Lemma~\ref{lemc3}, see Equation~\eqref{g}, we set $\varphi(y) = \psi(y) (1-y)^{-p}$ for $y\in [0,1)$.

\medskip

\noindent\textbf{Step~1.} If $a\in \mathcal{A}$ then $\ell(a)>0$,
from which we readily deduce that $\varphi(y)\to\infty$ as $y\to 1$,
and obviously $\sup\limits_{y\in[0,1)}\{\varphi(y)\} > \kappa$.
Conversely, if $\sup\limits_{y\in[0,1)}\{\varphi(y)\} > \kappa$,
then necessarily $\ell(a)\ne 0$ according to Lemma~\ref{lemc3} and
thus $a\in\mathcal{A}$.

\medskip

\noindent\textbf{Step~2.} We claim that $\mathcal{A}$ is non-empty. Indeed, assume for contradiction that $\mathcal{A}=\emptyset$, so that $\ell(a)=0$ for all $a>0$. We then infer from \eqref{c9}, \eqref{c11}, and the non-negativity of $\varphi$ that
$$
\varphi'(y) \ge \frac{p a^{2-p}}{p-1} (1-y)^{1-p} - \frac{p \kappa^{1/p}}{1-y} \varphi(y)^{(p-1)/p} \ge \frac{p a^{2-p}}{p-1} (1-y)^{1-p} - \frac{p \kappa}{1-y}
$$
for $y\in (0,1)$. Integrating over $(0,1/2)$ and using once more \eqref{c9} give
$$
\kappa \ge \varphi(1/2) \ge \frac{p a^{2-p}}{(p-1)(2-p)} \left( 1-2^{p-2} \right) - p \kappa \log{2}\ ,
$$
and a contradiction for $a$ large enough. Consequently, $\mathcal{A}$ is non-empty.

\medskip

\noindent\textbf{Step~3.} We put $a^* := \inf \mathcal{A}$. A
straightforward consequence of the characterization \eqref{c15} and
the monotonicity of $\psi(\cdot,a)$ with respect to $a$ established
in Lemma~\ref{lemc1} and \eqref{c15} is that $(a^*,\infty)\subset
\mathcal{A}$. Furthermore, if $a\in \mathcal{A}$, then $\ell(a)>0$
and it follows from Lemma~\ref{lemc1} that, for $\delta\in (0,a)$
$$
0<\ell(a)\le \ell(a+\delta) \;\text{ and }\; 0<\ell(a) \le \ell(a-\delta) + K(p) \delta^{2-p}\ .
$$
Therefore $\ell(a+\delta)>0$ and $\ell(a-\delta)>0$ for $\delta$
small enough, so that $(a-\delta,a+\delta)\subset \mathcal{A}$ for
$\delta>0$ small enough. In particular, $\mathcal{A}$ is open and
thus coincides with $(a^*,\infty)$.
\end{proof}

Concerning $\mathcal{C}$ one has the following result.

\begin{lemma}\label{lemc5}
The following statements are equivalent:
\begin{itemize}
\item[(c1)] $a \in \mathcal{C}$.
\item[(c2)] $\sup\limits_{y\in [0,1)} \left\{ \psi(y,a) (1-y)^{-p} \right\} < \kappa$.
\item[(c3)] The derivative $\varphi'(\cdot, a)$ of the function $\varphi(\cdot,a)$ defined in \eqref{g} vanishes at least once in
$(0,1)$.
\item[(c4)] There is $Y_a\in (0,1)$ such that
\begin{equation}
\varphi'(Y_a,a)=0\ , \quad \varphi'(y,a) (y-Y_a) < 0\ , \qquad y\in (0,1)\setminus\{Y_a\} \ . \label{c17}
\end{equation}
\end{itemize}
Furthermore there is $a_*>0$ such that $\mathcal{C}=(0,a_*)$.
\end{lemma}

\begin{proof}
Recall that $\varphi(y) = \psi(y) (1-y)^{-p}$ for $y\in [0,1)$, see Equation~\eqref{g}.

\medskip

\noindent\textbf{Step~1.} Assume first that
$\sup\limits_{y\in[0,1)}\{\varphi(y)\} < \kappa$. This property
readily implies that $\ell(a)=0$ and we deduce from
Lemma~\ref{lemc3} that the limit of $\varphi(y)$ as $y\to 1$ is
necessarily zero. Therefore $a\in\mathcal{C}$ and we have proved
that (c2) $\Rightarrow$ (c1).

Consider now $a\in\mathcal{C}$. Since $\varphi(0)=0$ by \eqref{c12} and $\varphi(y)\to 0$ as $y\to 1$, a generalization of Rolle's theorem guarantees that $\varphi'$ vanishes at least once in $(0,1)$, and (c1) $\Rightarrow$ (c3).

Assume next that $\varphi'$ vanishes at least once in $(0,1)$ and
denote its smallest zero by $Y_a\in (0,1)$. Since $\varphi'(0)=p
a^{2-p}/(p-1)>0$ by \eqref{c11}, the function
$\varphi'$ is positive in $[0,Y_a)$ and it follows from \eqref{c11}
that
$$
\varphi''(Y_a) = - \frac{p (2-p) a^{2-p}}{p-1} (1-Y_a)^{-p} < 0\ .
$$
Consequently, $\varphi'$ is negative in a right neighborhood of $Y_a$. Assume for contradiction that there is $Y_1\in (Y_a,1)$ such that $\varphi'(y)<0$ for $y\in (Y_a,Y_1)$ and $\varphi'(Y_1)=0$. Then $\varphi''(Y_1)\ge 0$ while \eqref{c11} implies that $\varphi''(Y_1) = - p (2-p) a^{2-p} (1-Y_1)^{-p}/(p-1)<0$, and a contradiction.  Therefore $\varphi'<0$ in $(Y_a,1)$ and we have shown that $\varphi$ enjoys the property \eqref{c17}, that is, (c3) $\Rightarrow$ (c4).

Finally, assume that $\varphi$ satisfies \eqref{c17}. Then
$\sup\limits_{y\in[0,1)}\{\varphi(y)\}=\varphi(Y_a)$ and we deduce
from \eqref{c11} that
$$
\frac{p}{1-Y_a} \varphi(Y_a)^{(p-1)/p} \left[ \kappa^{1/p} - \varphi(Y_a)^{1/p} \right] = \frac{p a^{2-p}}{p-1} (1-Y_a)^{1-p}>0\ .
$$
Consequently $\varphi(Y_a)<\kappa$ and (c4) $\Rightarrow$ (c2).

\medskip

\noindent\textbf{Step~2.} We now check that $\mathcal{C}$ is non-empty. To this end, consider $a>0$ such that
$$
a^{2-p} \le (p-1)^{p-1}/p^p = \max_{A\in (0,1)}\left\{ A^{(p-1)/p} - A \right\}\ .
$$
We fix $A\in (0,1)$ such that $A^{(p-1)/p} - A \ge a^{2-p}$ and set $\Sigma_A(y) = A (1-y)^{p/(p-1)}$ for $y\in [0,1)$. On the one hand,
\begin{align*}
\Sigma_A'(y) + \frac{p}{p-1} \Sigma_A(y)^{(p-1)/p} & = \frac{p}{p-1} (1-y) \left[ A^{(p-1)/p} - A (1-y)^{(2-p)/(p-1)} \right] \\
& \ge \frac{p}{p-1} (1-y) \left[ A^{(p-1)/p} - A \right] \\
& \ge \frac{p}{p-1} (1-y) a^{2-p} = \psi'(y) + \frac{p}{p-1} \psi(y)^{(p-1)/p}
\end{align*}
for $y\in (0,1)$. On the other hand, $\Sigma_A(0)=A>0=\psi(0)$. We are then in a position to apply Lemma~\ref{lemc0} with $(\xi_1,\xi_2)=(\psi,\Sigma_A)$ to conclude that
$$
0 \le \psi(y) \le A (1-y)^{p/(p-1)}\ , \qquad y\in [0,1)\ .
$$
Since $p<p/(p-1)$, the above estimate ensures that $a\in
\mathcal{C}$ and we have thus shown that $\mathcal{C}$ is non-empty
and contains the interval $\left( 0, (p-1)^{(p-1)/(2-p)}
p^{-p/(2-p)} \right]$.

\medskip

\noindent\textbf{Step~3.} Introducing $a_*:=\sup\{\mathcal{C}\}>0$,
we infer from the monotonicity of $\psi(\cdot,a)$ with respect to
$a$ (Lemma~\ref{lemc1}) that $(0,a_*)\subset \mathcal{C}$.

Assume for contradiction that $a_*\in\mathcal{C}$. Owing to \eqref{c17} there are $\delta>0$ and $\varepsilon>0$ such that
\begin{equation}
\varphi'(Y_{a_*}+\delta,a_*) < -2\varepsilon < \varphi'(Y_{a_*},a_*) = 0 < 2\varepsilon < \varphi'(Y_{a_*}-\delta,a_*) \label{c18}
\end{equation}
and
\begin{equation}
\frac{Y_{a_*}}{2} \le Y_{a_*} - \delta < Y_{a_*} < Y_{a_*} + \delta \le \frac{1+Y_{a_*}}{2}\ . \label{c19}
\end{equation}
Thanks to Lemma~\ref{lemc1}, $\varphi'(\cdot,a)$ depends
continuously on $a$ on $[0,(1+Y_{a_*})/2]$ and we infer from
\eqref{c18} and \eqref{c19} that there is $\alpha>0$ small enough
such that
$$
\varphi'(Y_{a_*}+\delta,a) < -\varepsilon < \varepsilon < \varphi'(Y_{a_*}-\delta,a)\ , \qquad a\in [a_*-\alpha,a_*+\alpha]\ .
$$
In particular, for all $a\in [a_*-\alpha,a_*+\alpha]$, the function
$\varphi'(\cdot,a)$ has a zero inside the interval
$(Y_{a_*}-\delta,Y_{a_*}+\delta)$. According to (c3), this means
that $[a_*-\alpha,a_*+\alpha]\subset\mathcal{C}$, which contradicts
the definition of $a_*$. Therefore $a_*\not\in \mathcal{C}$ and
$\mathcal{C} = (0,a_*)$.
\end{proof}

We finally turn to the description of the set $\mathcal{B}$ and show that it is a singleton.

\begin{proposition}\label{propc6}
There holds $a_*=a^*$ and $\mathcal{B}=\{a_*\}$, where $a^*$ and $a_*$ are defined in Lemma~\ref{lemc4} and Lemma~\ref{lemc5}, respectively.
\end{proposition}

\begin{proof}
Owing to Lemma~\ref{lemc4} and Lemma~\ref{lemc5} there holds
$\mathcal{B}=[a_*,a^*]$ and $\varphi'(\cdot,a)>0$ in $(0,1)$ for
$a\in\mathcal{B}$, recalling that the function $\varphi(\cdot,a)$ is
defined by \eqref{g}. Introducing $G:=\varphi(\cdot,a^*) -
\varphi(\cdot,a_*)$ it follows from Lemma~\ref{lemc1} and
\eqref{c11} that $G\ge 0$ and
\begin{align}
G'(y) & + \frac{p \kappa^{1/p}}{1-y} \left[ \varphi(y,a^*)^{(p-1)/p} - \varphi(y,a_*)^{(p-1)/p} \right] \nonumber\\
& = p\frac{G(y)}{1-y} + \frac{p}{p-1} \left[ (a^*)^{2-p} - (a_*)^{2-p} \right] (1-y)^{1-p} \label{c20}
\end{align}
for $y\in (0,1)$. Since $a_*\in\mathcal{B}$ we deduce from the definition of $\mathcal{B}$ that there is $Y\in (0,1)$ such that
$$
\varphi(y,a_*) \ge \left( p - \frac{1}{2} \right)^{-p}\ , \qquad y\in [Y,1)\ .
$$
Therefore, for $y\in [Y,1)$,
\begin{align*}
\varphi(y,a^*)^{(p-1)/p} - \varphi(y,a_*)^{(p-1)/p} & = \frac{p-1}{p}\ \int_{\varphi(y,a_*)}^{\varphi(y,a^*)} z^{-1/p}\ dz \\
& \le \frac{p-1}{p} \varphi(y,a_*)^{-1/p} G(y) \\
& \le \frac{(p-1)(2p-1)}{2p} G(y)\ .
\end{align*}
Combining the above estimate with \eqref{c20} gives, for $y\in [Y,1)$,
\begin{equation*}
\begin{split}
G'(y)+\frac{p}{(p-1)(1-y)}\frac{(p-1)(2p-1)}{2p}G(y)&\geq\frac{pG(y)}{1-y}+\frac{p}{p-1}\left[(a^*)^{2-p}-(a_*)^{2-p}\right](1-y)^{1-p}\\
&\geq \frac{pG(y)}{1-y},
\end{split}
\end{equation*}
whence, after easy manipulations,
$$
G'(y)\geq\frac{G(y)}{2(1-y)}, \quad y\in[Y,1).
$$
Integrating the above differential inequality on $[Y,y)$ for some
$y\in(Y,1)$, we find
\begin{equation}
G(y) \ge G(Y) \sqrt{\frac{1-Y}{1-y}}\ , \qquad y\in (Y,1)\ .
\label{c21}
\end{equation}
Assume now for contradiction that $a^*>a_*$. We deduce from Lemma
\ref{lemc1} and the fact that $Y\in(0,1)$ that
$\varphi(Y,a^*)>\varphi(Y,a_*)$, that is, $G(Y)>0$. It then follows
from \eqref{c21} that $G(y)\to\infty$ as $y\to 1$. However, the
definition of $\mathcal{B}$ entails that $G(y)\to 0$ as $y\to 1$,
clearly in contradiction with the previous assertion. Therefore
$a_*=a^*$ and the proof of Proposition~\ref{propc6} is complete.
\end{proof}

\subsection{Refined asymptotics as $y\to 1$ for $a\in\mathcal{C}$}\label{s3.4}

The final step is to identify the behavior of $\psi(y,a)$ as $y\to 1$ for $a\in \mathcal{C}$.

\begin{lemma}\label{lemc7}
If $a\in\mathcal{C}$ then
$$
\lim_{y\to 1} \psi(y,a) (1-y)^{-p/(p-1)} = a^{p(2-p)/(p-1)}\ .
$$
\end{lemma}

\begin{proof}
Let $a\in\mathcal{C}$.

\medskip

\noindent\textbf{Step~1.} We first prove that there exists
$M>a^{p(2-p)/(p-1)}$ such that
\begin{equation}
\psi(y) \le M (1-y)^{p/(p-1)}\ , \qquad y\in [0,1)\ . \label{c22}
\end{equation}
Indeed, let $\varepsilon\in (0,1)$ to be determined later and define
$$
\sigma_\varepsilon(y) := \frac{1}{2 \varepsilon^{p(2-p)/(p-1)}} (1-y)^{p/(p-1)}\ , \qquad y\in (0,1)\ .
$$
Owing to the definition of $\mathcal{C}$, there is
$\bar{\varepsilon}\in (0,1)$ such that $\psi(y)\le (1-y)^p/2$ for
$y\in (1-\bar{\varepsilon},1)$. On the one hand, if $\varepsilon\in
(0,\bar{\varepsilon})$, there holds
$$
\sigma_\varepsilon(1-\varepsilon) = \frac{\varepsilon^p}{2} \ge \psi(1-\varepsilon)\ .
$$
On the other hand, for $y\in (1-\varepsilon,1)$,
\begin{align*}
\sigma_\varepsilon'(y) + \frac{p}{p-1} \sigma_\varepsilon(y)^{(p-1)/p} & = \frac{p}{p-1} (1-y) \left[ \frac{1}{2^{(p-1)/p} \varepsilon^{2-p}} - \frac{(1-y)^{(2-p)/(p-1)}}{2 \varepsilon^{p(2-p)/(p-1)}}  \right] \\
& \ge \frac{p}{p-1} (1-y) \frac{2^{1/p} - 1}{2 \varepsilon^{2-p}} \\
& \ge \frac{p a^{2-p}}{p-1} (1-y) = \psi'(y) + \frac{p}{p-1} \psi(y)^{(p-1)/p}\ ,
\end{align*}
as soon as
\begin{equation}
\frac{2^{1/p} - 1}{2 \varepsilon^{2-p}} \ge a^{2-p}\ . \label{c22a}
\end{equation}
We next choose $\varepsilon\in (0,\bar{\varepsilon})$ satisfying
\eqref{c22a}. This allows us to apply Lemma~\ref{lemc0} with
$(\xi_1,\xi_2) = (\psi,\sigma_\varepsilon)$ in order to obtain that
$\psi(y)\le \sigma_\varepsilon(y)$ for $y\in (0,1-\varepsilon)$.
This inequality extends to the whole interval $(0,1)$, possibly
taking a smaller value of $\varepsilon$.

\medskip

\noindent\textbf{Step~2.} The goal of this step is to improve
\eqref{c22}. To this end, fix $A\in \left( a^{p(2-p)/(p-1)}, M
\right)$ and $\varepsilon \in (0,1)$ such that
\begin{equation}
\varepsilon^{(2-p)/(p-1)} < \frac{A^{(p-1)/p} - a^{2-p}}{2M}\ . \label{c23}
\end{equation}
We define
$$
\tau(y) := \left( A + \frac{M-A}{\varepsilon} (1-y) \right) (1-y)^{p/(p-1)}\ , \qquad y\in (0,1)\ ,
$$
and deduce from \eqref{c22} that
$$
\tau(1-\varepsilon) = M \varepsilon^{p/(p-1)} \ge \psi(1-\varepsilon)\ .
$$
In addition, we infer from \eqref{c4} and \eqref{c23} that, for
$y\in (1-\varepsilon,1)$,
\begin{align*}
& \tau'(y) + \frac{p}{p-1} \tau(y)^{(p-1)/p} \\
& \qquad \ge  \frac{p}{p-1} (1-y) \left[ A^{(p-1)/p} - A (1-y)^{(2-p)/(p-1)} - \frac{2p-1}{p} \frac{M-A}{\varepsilon} (1-y)^{1/(p-1)} \right] \\
& \qquad \ge  \frac{p}{p-1} (1-y) \left[ A^{(p-1)/p} - \left( A + \frac{2p-1}{p} (M-A) \right) \varepsilon^{(2-p)/(p-1)} \right] \\
& \qquad \ge \frac{p}{p-1} (1-y) \left[ A^{(p-1)/p} - 2M \varepsilon^{(2-p)/(p-1)} \right] \\
& \qquad \ge \frac{p a^{2-p}}{p-1} (1-y) = \psi'(y) + \frac{p}{p-1} \psi(y)^{(p-1)/p}\ .
\end{align*}
Applying Lemma~\ref{lemc0} with $(\xi_1,\xi_2)=(\psi,\tau)$ implies that $\psi(y) \le \tau(y)$ for $y\in (1-\varepsilon,1)$. Consequently,
$$
\frac{\psi(y)}{(1-y)^{p/(p-1)}} \le A + \frac{M-A}{\varepsilon}\ (1-y)\ , \qquad y\in (1-\varepsilon,1)\ ,
$$
from which we deduce that
$$
\limsup_{y\to 1} \frac{\psi(y)}{(1-y)^{p/(p-1)}} \le A\ .
$$
As $A$ is arbitrarily chosen in $\left( a^{p(2-p)/(p-1)} , M \right)$, we end up with
$$
\limsup_{y\to 1} \frac{\psi(y)}{(1-y)^{p/(p-1)}} \le a^{p(2-p)/(p-1)}\ .
$$
Since
$$
\liminf_{y\to 1} \frac{\psi(y)}{(1-y)^{p/(p-1)}} \ge a^{p(2-p)/(p-1)}
$$
by \eqref{c8b}, the claimed result follows.
\end{proof}

\section{Proof of Theorem~\ref{thm1}}\label{s4}

We now undo the transformation \eqref{c1} and interpret the outcome
of Section~\ref{s3} in terms of $f(\cdot,a)$. Let $a\in (0,\infty)$.
It follows from \eqref{b1} and \eqref{c2} that
\begin{equation}
f'(r) = - a \psi\left( 1 - \frac{f(r)}{a} \right)^{1/p}\ , \qquad r\in [0,R(a))\ . \label{d0}
\end{equation}
Since $\psi(y) \sim p a^{2-p} y/(p-1)$ as $y\to 0$ and $p>1$, the function $z\mapsto \psi(1-z)^{-1/p}$ defined on $(0,1)$ belongs to $L^1(z_0,1)$ for all $z_0>0$. We may thus integrate \eqref{d0} and find
\begin{equation}
\int_{f(r)/a}^1 \frac{dz}{\psi(1-z)^{1/p}} = r \ , \qquad r\in [0,R(a))\ . \label{d1}
\end{equation}

\medskip

\noindent\textbf{Case~1: $a\in\mathcal{A}$.} According to the definition of $\mathcal{A}$, $\psi(y)$ has a positive limit $\ell(a)>0$ as $y\to 1$ and the function $z\mapsto \psi(1-z)^{-1/p}$ actually belongs to $L^1(0,1)$. We then deduce from \eqref{d1} that
$$
\int_0^1 \frac{dz}{\psi(1-z)^{1/p}} = R(a)\ ,
$$
that is, $R(a)<\infty$. Furthermore, $f'(R(a)) = - a \ell(a)^{1/p}<0$ by \eqref{d0} and the proof of Theorem~\ref{thm1}~(a) is complete.

\medskip

\noindent\textbf{Case~2: $a\in\mathcal{B}$.} By Proposition~\ref{propc6} there holds $a=a_*$ and the definition of $\mathcal{B}$ ensures that $\psi(1-z)^{1/p} \sim z/(p-1)$ as $z\to 0$. Therefore $z\mapsto \psi(1-z)^{-1/p}$ does not belong to $L^1(0,1)$ and we infer from \eqref{d1} that $R(a_*)=\infty$ and
$$
r \sim - (p-1) \log{(f(r))} \;\;\text{ as }\;\; r\to \infty\ .
$$
In particular, there is $R>0$ such that
$$
- \frac{p-1}{r} \log{(f(r))} \ge 1 - \frac{2-p}{2} = \frac{p}{2}\ , \qquad r\ge R\ ,
$$
from which we deduce that
\begin{equation}
\int_R^\infty e^r f(r)\ dr \le \int_R^\infty e^{-(2-p)r/2(p-1)}\ dr < \infty\ , \label{d2}
\end{equation}
since $p\in (1,2)$. Recalling \eqref{b1}, it follows from
\eqref{b1b} after integration that
$$
- e^r |f'(r)|^{p-1} = - \int_0^r e^\sigma f(\sigma)\ d\sigma\ ,
$$
which, together with \eqref{d2}, guarantees that $e^r |f'(r)|^{p-1}$ has a finite limit as $r\to\infty$ and
$$
\lim_{r\to \infty} e^r |f'(r)|^{p-1} = I := \int_0^\infty e^r f(r)\ dr\ .
$$
We then infer from the above property, \eqref{d0}, and the behavior of $\psi(y)$ as $y\to 1$ that
$$
f'(r) \sim - I^{1/(p-1)} e^{-r/(p-1)} \;\;\text{ and }\;\; f'(r) \sim - \frac{f(r)}{p-1} \;\;\text{ as }\;\; r\to \infty\ ,
$$
so that $f(r) \sim (p-1) I^{1/(p-1)} e^{-r/(p-1)}$ as $r\to \infty$. We have thus proved Theorem~\ref{thm1}~(b).

\medskip

\noindent\textbf{Case~3: $a\in\mathcal{C}$.} In that case,
$\psi(1-z)^{1/p} \sim a^{(2-p)/(p-1)} z^{1/(p-1)}$ as $z\to 0$ by
Lemma~\ref{lemc7}. Since $p\in (1,2)$ the function $z\mapsto
\psi(1-z)^{-1/p}$ does not belong to $L^1(0,1)$ and we infer from
\eqref{d1} that $R(a)=\infty$ and
$$
\frac{p-1}{2-p} \left( \frac{f(r)}{a} \right)^{-(2-p)/(p-1)} \sim a^{(2-p)/(p-1)} r \;\;\text{ as }\;\; r\to \infty\ ,
$$
hence Theorem~\ref{thm1}~(c).

\section*{Acknowledgments}

R. G. I. is supported by the Severo Ochoa Excellence project
SEV-2015-0554 (MINECO, Spain). Part of this work has been completed
while R. G. I. was enjoying a one-month ``Invited Professor" stay at
the Institut de Math\'ematiques de Toulouse, and he thanks for the
hospitality and the support.

\bibliographystyle{siam}
\bibliography{CriticalExtinction1d}

\begin{thebibliography}{10}

\bibitem{CQW03}
{\sc X.~Chen, Y.~Qi, and M.~Wang}, {\em Self-similar singular solutions of a
  {$p$}-{L}aplacian evolution equation with absorption}, J. Differential
  Equations, 190 (2003), pp.~1--15.

\bibitem{dPSa02}
{\sc M.~del Pino and M.~S{\'a}ez}, {\em Asymptotic description of vanishing in
  a fast-diffusion equation with absorption}, Differential Integral Equations,
  15 (2002), pp.~1009--1023.

\bibitem{FeVa01}
{\sc R.~Ferreira and J.~L. V{\'a}zquez}, {\em Extinction behaviour for fast
  diffusion equations with absorption}, Nonlinear Anal., 43 (2001),
  pp.~943--985.

\bibitem{IaLa12}
{\sc R.~G. Iagar and {\relax Ph}.~Lauren{\c{c}}ot}, {\em Positivity, decay, and
  extinction for a singular diffusion equation with gradient absorption}, J.
  Funct. Anal., 262 (2012), pp.~3186--3239.

\bibitem{IaLa13a}
\leavevmode\vrule height 2pt depth -1.6pt width 23pt, {\em Eternal solutions to
  a singular diffusion equation with critical gradient absorption},
  Nonlinearity, 26 (2013), pp.~3169--3195.

\bibitem{IaLa13b}
\leavevmode\vrule height 2pt depth -1.6pt width 23pt, {\em Existence and
  uniqueness of very singular solutions for a fast diffusion equation with
  gradient absorption}, J. Lond. Math. Soc. (2), 87 (2013), pp.~509--529.

\bibitem{IaLaxx}
\leavevmode\vrule height 2pt depth -1.6pt width 23pt, {\em Self-similar
  extinction for a diffusive {H}amilton-{J}acobi equation with critical
  absorption}.
\newblock preprint, 2016.

\bibitem{Kw89}
{\sc M.~K. Kwong}, {\em Uniqueness of positive solutions of {$\Delta
  u-u+u^p=0$} in {${\bf R}^n$}}, Arch. Rational Mech. Anal., 105 (1989),
  pp.~243--266.

\bibitem{SeTa00}
{\sc J.~Serrin and M.~Tang}, {\em Uniqueness of ground states for quasilinear
  elliptic equations}, Indiana Univ. Math. J., 49 (2000), pp.~897--923.

\bibitem{Shi04}
{\sc P.~Shi}, {\em Self-similar very singular solution of a {$p$}-{L}aplacian
  equation with gradient absorption: existence and uniqueness}, J. Southeast
  Univ. (English Ed.), 20 (2004), pp.~381--386.

\bibitem{ShWa16}
{\sc N.~Shioji and K.~Watanabe}, {\em Uniqueness and nondegeneracy of positive
  radial solutions of {${\rm div}(\rho\nabla u)+\rho(-gu+hu^p)=0$}}, Calc. Var.
  Partial Differential Equations, 55 (2016), p.~55:32.

\bibitem{Ya91b}
{\sc E.~Yanagida}, {\em Uniqueness of positive radial solutions of {$\Delta
  u+g(r)u+h(r)u^p=0$} in {${\bf R}^n$}}, Arch. Rational Mech. Anal., 115
  (1991), pp.~257--274.

\bibitem{YeYi15}
{\sc H.~Ye and J.~Yin}, {\em Uniqueness of self-similar very singular solution
  for non-{N}ewtonian polytropic filtration equations with gradient
  absorption}, Electron. J. Differential Equations,  (2015), pp.~No. 83, 9.

\end{thebibliography}

\end{document}